\documentclass[a4paper,11pt,english]{smfart}
\usepackage[latin1]{inputenc}
\usepackage[francais,english]{babel}
\usepackage{amsfonts}
\usepackage{amsmath} 
\usepackage{latexsym}
\usepackage{array}
\usepackage{amssymb}
\usepackage{textcomp}
\usepackage{pifont}
\usepackage{enumerate}
\usepackage{smfthm}
\usepackage{graphicx}
\usepackage{color}
\usepackage[T1]{fontenc}
\usepackage{hyperref} 
\theoremstyle{plain}

\newtheorem{assumption}{Assumption}

\newcommand{\ligne}{\vspace{1\baselineskip}}
\newcommand{\ph}{\phantomsection}

\newcommand{\R}{  \mathbb{R}   }

\newcommand{\X}{  \mathcal{X}  }
\newcommand{\eps}{\varepsilon}
\newcommand{\e}{  \text{e}   }

\newcommand{\N}{  \mathbb{N}   }

\newcommand{\K}{\mathcal{K}}

\newcommand{\dis}{\displaystyle}
\newcommand{\h}{  h }

\renewcommand{\a}{  \alpha   }

\renewcommand{\phi}{  \varphi  }
\renewcommand{\L}{  \mathcal{L}   }

\renewcommand{\>}{  \rangle   }
\renewcommand{\l}{  \ell  }

\newcommand{\dist}{\operatorname{dist}}

\numberwithin{equation}{section}

 \author{ Beno\^it Gr\'ebert}
\address{Laboratoire de Math\'ematiques J. Leray, Universit\'e de Nantes, UMR CNRS 6629\\
2, rue de la Houssini\`ere \\
44322 Nantes Cedex 03, France.}
\email{benoit.grebert@univ-nantes.fr}
 \author{Tiphaine J\'ez\'equel} 
\address{  IRMAR, ENS Cachan Bretagne, CNRS, UEB
Avenue Robert Schuman, 35170 Bruz, France.}
\email{tiphaine.jezequel@inria.fr}  
 \author{Laurent Thomann}
\address{Laboratoire de Math\'ematiques J. Leray, Universit\'e de Nantes, UMR CNRS 6629\\
2, rue de la Houssini\`ere \\
44322 Nantes Cedex 03, France.}
\email {laurent.thomann@univ-nantes.fr}
 
\title[Stability  of Klein-Gordon]
{Stability of large periodic solutions of Klein-Gordon near a homoclinic orbit}

\begin{document}
\frontmatter

\begin{abstract}
We consider the   Klein-Gordon equation  (KG) on a Riemannian surface $M$
$$  \partial^{2}_t u-\Delta u-m^{2}u+u^{2p+1} =0,\quad p\in \N^{*},\quad  (t,x)\in \R\times M,$$
which  is globally well-posed in the energy space. Viewed as a first order Hamiltonian system  in  the variables $(u, v\equiv \partial_t u)$, the associated flow  lets invariant the two dimensional space of $(u,v)$ independent of $x$. It turns out that in this  invariant space,  there is a homoclinic orbit to the origin, and a family of  periodic solutions inside the loops of the homoclinic orbit. In this paper we study the stability of these periodic orbits under the (KG) flow, i.e. when turning on the nonlinear interaction with the non stationary modes. By a shadowing method, we prove that around the periodic orbits, solutions stay close to them during a time of order $(\ln{\eta})^2$, where $\eta$ is the distance between the periodic orbit considered and the homoclinic orbit.
    \end{abstract}

\keywords{Klein-Gordon equation, wave equation, homoclinic orbit.}
\altkeywords{  Equation de Klein-Gordon, \'equation des ondes, orbite homocline.}\frontmatter
\subjclass{ 37K45, 35Q55, 35Bxx}
\thanks{
\noindent B.G. was supported in part by the  grant ANR-10-BLAN-DynPDE.\\
 B.G. and L.T.  were supported in part by the  grant  ``HANDDY'' ANR-10-JCJC 0109 \\
 and by the  grant  ``ANA\'E'' ANR-13-BS01-0010-03}

\maketitle

\section{Introduction, statement of the main results}
\subsection{General introduction} Denote by  $M$   a compact Riemannian manifold without boundary of dimension 1,2 or 3 and denote by $\Delta=\Delta_{M}$ the Laplace-Beltrami operator on $M$. In this paper we are concerned with the following nonlinear Klein-Gordon  (KG) equation
\begin{equation}\label{kg} \tag{KG}
\left\{
\begin{aligned}
& \partial^{2}_t u-\Delta u-m^{2}u+u^{2p+1} =0,\quad 
(t,x)\in\R\times M,\\
&u(0,x)= u_{0}(x),\quad \partial_{t}u(0,x)= u_{1}(x),
\end{aligned}
\right.
\end{equation} 
where $p\geq 1$ is an integer, and $(u_{0},u_{1})\in H^{1}(M)\times L^{2}(M)$ are real-valued.\ligne

Let us recall  that there exists a Hilbert basis of $L^{2}(M)$ composed with eigenfunctions $(e_{n})_{n\geq 0}$ of~$\Delta$. Moreover,  there exists a sequence $0=\lambda_{0}<\lambda_{1}\leq \dots \leq \lambda_{n}\leq \dots$ so that  
 \begin{equation*}
 -\Delta e_{n}=\lambda^{2}_{n}e_{n}, \quad n\geq 0.
 \end{equation*}

We make following  assumptions 
\begin{assumption}\ph\label{assumption0}
The parameter  $m$ satisfies $\dis 0<m<\lambda_{1}.$
\end{assumption}

\begin{assumption}\ph\label{assumption}
The manifold $M$ and the integer $p$ satisfy either:\\[5pt]
\indent $\bullet$ $M$ is any compact manifold without boundary of dimension 1 or 2 and $p\geq 1$\\[5pt]
\indent $\bullet$ $M$ is any compact manifold without boundary of dimension 3 and $p= 1$.\\[5pt]
Moreover, up to a rescaling, we can assume that ${\rm Vol}\, M=1$.
\end{assumption}

Let us recall the well posedness result  proved in \cite{GJT} 

\begin{prop}[\cite{GJT}, Theorem 2.2]
Under Assumption \eqref{assumption0}, the equation \eqref{kg} is globally well posed in the energy space $H^1(M)\times L^2(M)$.
\end{prop}

Under Assumptions \ref{assumption0} and \ref{assumption}, the stationary solutions of (KG)  (solutions which only depend on the space variable)  are exactly the constants $u=0$, $u=m^{1/p}$ and $u=-m^{1/p}$. The origin is an equilibrium with an unstable direction. In fact, the eigenvalues of $-\Delta-m^{2}$ are the $(\lambda_{k}^{2}-m^{2})_{k\in \N}$. Since~$0<m<\lambda_{1}$,  the case $k=0$ only,  gives the hyperbolic directions, corresponding to the solution $\exp(mt)$ for $t>0$ (resp. $\exp(-mt)$ for $t<0$). It turns out that \eqref{kg} admits a homoclinic orbit  to the origin which is independent of $x$. In the previous work \cite{GJT} we have proved that (KG) admits a  family  of heteroclinic connections to the center manifold which are close to this homoclinic orbit. The expected picture is  that we have a tube of heteroclinic connections surrounding the homoclinic orbit (but the statement in \cite{GJT} is not so precise, only a large family of heteroclinic orbits is constructed).\ligne

As we pointed out in \cite{GJT}, the dynamics around the elliptic points $u=\pm m^{1/p}$ can be partially described by the KAM theory or the Birkhoff normal form theory.  The KAM theory gives the existence of a large family of finite dimensional invariant tori close to the equilibrium (see \cite{Wayne, Posch,kuk} or the book~\cite{kuk}). The Birkhoff normal form approach gives the stability during polynomial times for any initial condition close to the equilibrium  (see \cite{BamGre}).\ligne

 In this work our point of view is different: (KG) also admits large periodic orbits inside the loops of the homoclinic orbit (see Figure 1) and we are interested  in the stability of these large  periodic orbits. Observe that they are not close to the origin since they turn around $u=m^{1/p}$.  By a shadowing method, we prove that around the periodic orbits, solutions stay close to them during a time of order $(\ln{\eta})^2$, where $\eta$ is the distance between the periodic orbit considered and the homoclinic orbit. Actually, thanks to an energy method, it is easy to get a control for times of order $\ln \eta$, which is the typical timescale in the presence of a hyperbolic point. In our context, $\ln \eta$ is the timescale needed to achieve one loop, and our contribution consists in proving that we can follow the solution for $\ln \eta $ loops. The two main ingredients used in the proof are 
 \begin{itemize}
   \item[$\bullet$] The Hamiltonian is negative on the trajectory, which implies a confinement of the solution;
 \item[$\bullet$] The trajectory is close to the homoclinic orbit, which gives the pattern of the solution.
 \end{itemize}
Combining these two  facts, we  conclude with a bootstrap argument.

\subsection{Hamiltonian structure of \eqref{kg}} 
 As in \cite{GJT}, we define the scalar product on $L^{2}(M)$   by $\dis\<f,g\>=\frac1{{\rm Vol}\, M}\int_{M}fg$, where ${\rm Vol}\, M$ denotes  the volume of $M$, we assume that $\|e_{n}\|_{L^{2}}=1$ and we set $e_{0}=1$.   \ligne
 
Denote by $v=\partial_{t}u$ and introduce
\begin{equation}\label{Hami0}
H=\frac12\int_{M}\Big(|\nabla_{x}u|^{2}+v^{2}-m^{2}u^{2}\Big)+\frac1{2p+2}\int_{M}u^{2p+2}.
\end{equation}
Then, the system \eqref{kg} is equivalent to 
\begin{equation}\label{ham}
\dot{u}=\frac{\delta H}{\delta v},\quad \dot{v}=-\frac{\delta H}{\delta u}.
\end{equation}
We write 
\begin{equation*}
u(t,x)=\sum_{n=0}^{+\infty}a_{n}(t)e_{n}(x),\quad  v(t,x)=\sum_{n=0}^{+\infty}b_{n}(t)e_{n}(x),
\end{equation*}
where 
\begin{align*}
(a_n)_{n\in\N}\in  h^1(\N,\R)&:=\big\{ x=(x_n)_{n\in\N} \mid \|x\|^{2}_{h^1}=\sum_{n\in\N}(1+\lambda_{n}^2)|x_n|^2 <+\infty \big\},\\
(b_n)_{n\in\N}\in  \l^2(\N,\R)&:=\big\{ x=(x_n)_{n\in\N} \mid \|x\|^{2}_{\l^2}=\sum_{n\in\N}|x_n|^2 <+\infty \big\},
\end{align*}
in such a way that to the continuous phase space $\X:=H^{1}\times L^{2}$ corresponds the discrete one  $h^1\times \l^2$. We endow this space with the natural norm and distance
\begin{align*}
\|X\|_\X&=\|u\|_{h^{1}}+\|v\|_{\ell^{2}},\quad \text{for}\quad  X=(u,v)\\
\dist_\X (X,Y)&= \|X-Y\|_\X, \quad \text{for}\quad  X,Y \in \X.
\end{align*}
In the coordinates $(a_{n},b_{n})_{n\geq0}$, the Hamiltonian in \eqref{Hami0} reads 
\begin{equation}\label{Hami}
H=\frac12\sum_{n=0}^{+\infty}\Big[(\lambda_{n}^{2}-m^{2})a^{2}_{n}+b^{2}_{n}\Big]+\frac1{2p+2}\int_{M}\Big(\sum_{k=0}^{+\infty}a_{k}e_{k}(x)\Big)^{2p+2}\text{d}x,
\end{equation}
and the system \eqref{ham} becomes
 \begin{equation}\label{syst} 
\left\{
\begin{aligned}
&\dot{a}_{n} =b_{n},\quad n\geq 0\\
&\dot{b}_{n} =-(\lambda_{n}^{2}-m^{2})a_{n}-\int_{M}\Big(\sum_{k=0}^{+\infty}a_{k}e_{k}(x)\Big)^{2p+1}e_{n}(x)\text{d}x,\quad n\geq 0.
\end{aligned}
\right.
\end{equation}

\subsection{Space-stationary solutions: homoclinic orbit and family of periodic orbits}\label{Sect.14}
The space-stationary solutions of \eqref{kg} exactly correspond to the  solutions of \eqref{syst} satisfying $a_{n}=b_{n}=0$ for $n\geq 1$. In this case, the equation on $(a_{0},b_{0})$ reads 
 \begin{equation} \label{toy0}
\left\{
\begin{aligned}
&\dot{a}_{0} =b_{0}\\
&\dot{b}_{0} =m^{2}a_{0}- a^{2p+1}_{0}, 
\end{aligned}
\right.
\end{equation} 
and this system possesses a homoclinic solution to 0, which  we will denote in the sequel by $\h: t\mapsto (\alpha(t),\beta(t))$. We  denote by $\mathcal{K}_{0}\subset h^1\times\ell^2$ the curve which is described by $a_0(t)=h(t),\ b_0(t)=h'(t)$ and $a_n(t)=b_n(t)=0$ for $n\geq 1$ (see Figure \ref{PhasePortraita0b0}). 
\begin{figure}[h!]
\centering 
\def\svgwidth{100mm}
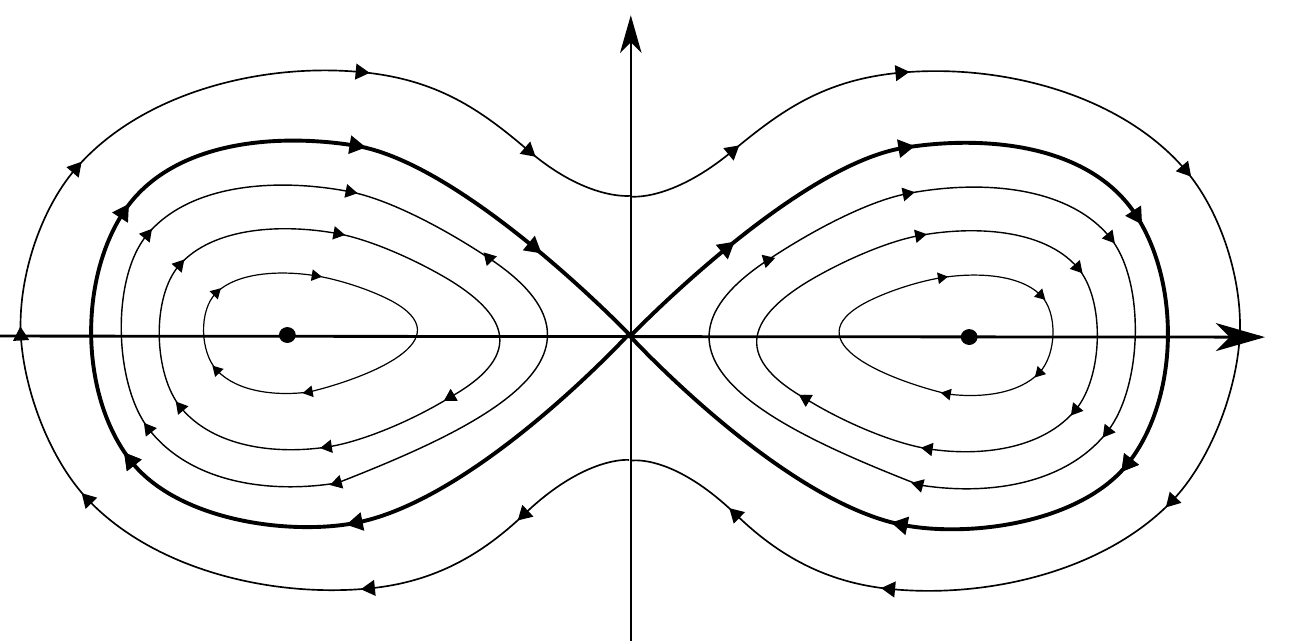
\caption{Phase portrait for the space-stationary set $a_{n}=b_{n}=0$ for $n\geq 1$.}
\label{PhasePortraita0b0}
\end{figure} 
Indeed we can explicitly compute 
\begin{equation}\label{def.homo}
h(t)=\frac{m^{1/p}(p+1)^{1/(2p)}}{\big(\cosh(pmt)\big)^{1/p}},\quad h'(t)=-m^{1/p+1}(p+1)^{1/(2p)}\frac{\sinh(pmt)}{\big(\cosh(pmt)\big)^{1/p+1}}.
\end{equation}

For $\eta>0$ denote by $\mathcal{K}_{\eta}$ the trajectory of \eqref{syst} given by the  initial conditions $a_{0}(0)=\eta$, $b_{0}(0)=0$, $a_n(0)=b_n(0)=0$ for $n\geq 1$ (see Figure \ref{PhasePortraita0b0}). This orbit remains in the plane $\{ (a,b)\in h^1\times \ell^2 \mid a_n=b_n=0\ \text{for}\ n\geq 1\}$ and is periodic with a  period  of order $\dis \ln \frac1{\eta}$ (see below the proof): most of that time is dedicated to cover the very principle and the very end of the loop, {\it i.e.} when $(a_0,b_0)$ is close to $(\eta,0)$. \ligne
  
Our main result states that for $\eta$ small enough these periodic solutions are stable during a long time of order $(\ln \eta)^2$ in the sense that if one starts $\eta^3$-close to $\mathcal{K}_{\eta}$ then one remains $\eta^{2}$-close to $\mathcal{K}_{\eta}$ during a time of order $(\ln \eta)^2$. More precisely 
  \begin{theo}\ph\label{thm4}
 Let $M$, $p$ and $m$ fixed and satisfying Assumptions \ref{assumption0} and \ref{assumption}. There exists  $\eta_0>0$, $0<c<1$ and $C>0$ such that if $0<\eta<\eta_0$ and if the initial datum $(a(0),b(0))$ satisfies  $$\dist_\X\Big((a(0), b(0)), \K_\eta\Big)\leq \eta^3$$  then the solution of \eqref{syst} satisfies
  \begin{enumerate}[(i)]
  \item  There exists $T_\eta>0$ satisfying  $c\ln \frac1{\eta} \leq T_\eta\leq \frac1c\ln \frac1{\eta}$ such that
$$\dist_\X\Big((a(T_\eta),b(T_\eta)),(a(0), b(0))\Big)\leq C\eta^{2}$$
  \item  For all $|t|\leq c \big(\ln\frac1{\eta}\big)^{2}$
  \begin{equation*}
 \dist_\X\Big(\big(a_{0}(t),b_{0}(t)\big), \mathcal{K}_{\eta}\Big)\leq C\eta^{2}.
 \end{equation*}
  \end{enumerate}
  \end{theo}

 This result shows that a trajectory issued from a tube of thickness $\eta^3$ centered on $\K_\eta$   describes in the mode 0 several loops around $(m^{\frac{1}{p}},0)$ and remains very small in the other modes. Indeed the time~$T_\eta$ is  the time necessary for $(a_0,b_0)$ to complete one loop around $(m^{\frac{1}{p}},0)$, and the previous result shows  the trajectory completes at least $c^{2}\ln\frac1{\eta}$ loops around $(m^{\frac{1}{p}},0)$, staying inside the tube of thickness~$\eta^{2}$ centered on $\K_\eta$.  
  
\begin{rema} Numerical simulations confirm this stability result and seem to show that it holds true for larger timescales. Moreover, it is likely that there exist initial conditions as in Theorem \ref{thm4} and which satisfy  $(ii)$ for all times. Such a result is out of reach with our method and could possibly be attacked with a KAM method.
 \end{rema}
 
 \subsection{On the linear case.} This stability problem is related to the more general question of existence of periodic or quasi-periodic solutions to (KG). This question is not evident even in the linear case. Actually, we can consider the linearized equation of (KG) around the periodic orbit $\K_\eta$ as a first approximation to understand the possible dynamics: take $a_0(t)$ as a periodic solution of the system~\eqref{toy0} inside the loops (see Figure \ref{PhasePortraita0b0}) and search for $w=\sum_{k\geq 1}a_k e_k$ as a (small) solution of the following linear wave equation
  \begin{equation}\label{lkg} 
\partial^{2}_t w-\Delta w-m^{2}w+V(t)w =0,\quad 
(t,x)\in\R\times M
\end{equation} 
where $t\mapsto V(t)=(2p+1) a^{2p}_0(t)$ is a time depending potential. In the discrete variables, the equation~\eqref{lkg} reads
\begin{equation} \label{toy2}
\left\{
\begin{aligned}
&\dot{a}_{n} =b_{n}\\
&\dot{b}_{n} =-(\lambda_{n}^2-m^{2})a_{n}-V(t)a_n
\end{aligned}
\quad \quad \text{for } n\geq 1. \right.
\end{equation} 
  The existence of small periodic (or quasi-periodic) solutions of \eqref{lkg} is related to the reducibility of this equation to the autonomous case and to the Floquet theory. In the case where $V$ is a small potential the question can be considered with the help of the KAM theory (see \cite{EK2} in the case of a linear Schr\"odinger equation on a torus or \cite{GT} in the case of a linear Schr\"odinger equation on the line with a harmonic potential). To the best of our knowledge there is no result in a case where the time depending linear perturbation is not small. With the same techniques, we can apply our result  to the linear equation \eqref{toy0},\eqref{toy2}, and this shows that there is at least  some stability around the  solutions~$\K_\eta$. 
 
\begin{rema}
 In our previous paper \cite{GJT}, we also consider the linearised system \eqref{toy2} but in the case where $t\mapsto a_{0}(t)$ describes the homoclinic orbit, and in which case the linear flow is bounded. We refer to \cite[Appendix]{GJT} where we crucially use that $\int_{0}^{+\infty}a_{0}(t)\text{d}t<+\infty$.
 
 \end{rema}

  \begin{enonce*}{Notations}
 In this paper $c,C>0$ denote constants the value of which may change
from line to line. These constants will always be universal, or depend on the fixed quantities $m$ and $p$.\\
We denote by $\N$ the set of the non negative integers, and $\N^{*}=\N\backslash\{0\}$. We set $\X=H^{1}(M)\times L^{2}(M)$.
 \end{enonce*}

\section{Proof of the main result}\label{Sect.2}

\subsection{Notations and strategy of the proof}

For $\dis u=\sum_{k=0}^{+\infty}a_{k}e_{k}$ we define  
\begin{equation}\label{defU}
 U:= \sum_{k=1}^{+\infty}a_{k}e_{k}, 
\end{equation} 
so that in particular, $u$ reads $u=a_0e_0+U$.
 
In the sequel we will use the following decomposition of the energy   $H$ (see \eqref{Hami})  
\begin{equation}\label{HamiJ}
H=\frac12(b^{2}_{0}-m^{2}a^{2}_{0})+J+\frac1{2(p+1)}\int_{M}\big(a_{0}+U\big)^{2p+2}\text{d}x=\frac12 b_0^2+\frac12 f(a_0)+J+r,
\end{equation}
 where 
 \begin{equation}\label{def.f}
f(x):=-m^{2}x^{2}+\frac1{p+1}x^{2p+2},
\end{equation}
\begin{equation}\label{defJ}
 J:=\frac12\sum_{k=1}^{+\infty}\Big[(\lambda_{k}^{2}-m^{2})a^{2}_{k}+b^{2}_{k}\Big],
\end{equation}
and 
\begin{equation}\label{def:r}
r(t)=\frac{1}{2(p+1)}\int_{M}\big((a_0+U)^{2p+2}-a_0^{2p+2}\big)\text{d}x.
\end{equation}
Note that $J\sim \|U\|^{2}_{\X}$. In Subsection \ref{sect24}, we  prove that $J$ remains small. As a consequence $r$ is also small since it is a quadratic quantity in $U$. Then, as long as $J$ is small, the dynamics of the solution is governed by the space stationary dynamics (see Subsection \ref{Sect.14}) given by the Hamiltonian $H_{0}=\frac12b^{2}_{0}+f(a_{0})$.\ligne

Moreover, we introduce the notations $q_{k}$ for $k\in\N$, defined by 
\begin{equation}\label{def.pk}
q_{k}=\int_{M}\Big((a_{0}+U)^{2p+1}-a^{2p+1}_{0}-(2p+1)a^{2p}_{0}U\Big)e_{k}\text{d}x,
\end{equation}
so that the systems \eqref{syst} read
\begin{equation}\label{syst0} 
\left\{
\begin{aligned}
&\dot{a}_{0} =b_{0}\\
&\dot{b}_{0} =-f'(a_0)-q_0,
\end{aligned}
\right.
\end{equation} 
\begin{equation}\label{systn} 
\left\{
\begin{aligned}
&\dot{a}_{k} =b_{k},\quad k\geq 1\\
&\dot{b}_{k} =-(\lambda_{k}^{2}-m^{2})a_{k}-(2p+1)a^{2p}_{0}a_{k}-q_{k},\quad k\geq 1.
\end{aligned}
\right.
\end{equation} 

We finally fix a real $\delta\in]0,m^{\frac1p}[$, and denote by $\delta'$ the unique real of $]m^{\frac1p},(p+1)^{\frac1{2p}}m^{\frac1p}[$ such that $f(\delta')=f(\delta)$.

\subsection{Preliminary results}
In the sequel we will need
\begin{lemm}\ph\label{lemsobo}
Let $U\in H^{1}(M)$. Then  \\[5pt]
$\bullet$ When $M$ has dimension 1 or 2, for all $2\leq q<+\infty$
\begin{equation*}
\|U\|_{L^{q}(M)}\leq C_{q} J^{1/2}.
\end{equation*}
$\bullet$ When $M$ has dimension 3, for all $2\leq q\leq 6$
\begin{equation*}
\|U\|_{L^{q}(M)}\leq C_{q} J^{1/2}.
\end{equation*}
 \end{lemm}
 
\begin{proof}
By Sobolev, in each of the previous cases, there exists $C_{q}>0$  so that for all $U\in H^{1}(M)$ we have $\|U\|_{L^{q}(M)}\leq C_{q} \|U\|_{H^{1}(M)}$ and the result follows.
\end{proof}
\begin{lemm}\ph\label{lem:energy}
Let $(a,b)\in \X$ such that  $\dist_\X\big( (a,b),\K_\eta\big)\leq \eta^3$ then
\begin{equation}\label{heta} H(a,b)=-\frac 1 2 m^2 \eta^2 +O(\eta^3).\end{equation}
\end{lemm}
\begin{rema}
In the following, we prove in Lemma \ref{recip} a reciprocal result in the case $J$ small.
\end{rema}
\begin{proof} Let $(\tilde a,\tilde b)\in \K_\eta$ such that $\|(a,b)-(\tilde a, \tilde b)\|_\X \leq \eta^3$. On the one hand, since we have $(\eta,0)\bigotimes_{n=1}^{+\infty}(0,0)\in\K_\eta$, then 
$$H(\tilde a, \tilde b)=H\Big((\eta,0)\bigotimes_{n=1}^{+\infty}(0,0)\Big)=-\frac 1 2 m^2 \eta^2.$$
On the other hand denoting by $\dis \tilde u=\sum_{k=0}^{+\infty}\tilde a_{k}e_{k}$ and $\tilde J=\frac12\sum_{k=1}^{+\infty}\Big[(\lambda_{k}^{2}-m^{2})\tilde a^{2}_{k}+\tilde b^{2}_{k}\Big]$
\begin{multline*}
\Big|H(a,b)-H(\tilde a, \tilde b)\Big| \leq \\
\begin{aligned}
&\leq \frac12|\tilde b^{2}-b_0^2|+\frac {m^{2}}{2}|\tilde a^2_0-a^{2}_{0}|+|J-\tilde J|+\frac1{2(p+1)}\int_{M}|u^{2p+2}-\tilde u^{2p+2}|\text{d}x.\\
&\leq C \| (a,b)-(\tilde a,\tilde b)\|_\X \Big(  \| (a,b)\|_\X + \| (\tilde a,\tilde b)\|_\X \Big)+C\|u-\tilde u\|_{L^{2p+2}(M)}   \big(\|u\|^{2p+1}_{L^{2p+2}(M)} +\|\tilde u\|^{2p+1}_{L^{2p+2}(M)} \big) \\
&\leq C \| (a,b)-(\tilde a,\tilde b)\|_\X =O(\eta^3).
\end{aligned}
\end{multline*}
\end{proof}

We now fix an initial condition satisfying $\dist_\X\Big( (a(0),b(0)),\K_\eta\Big)\leq \eta^3$ with $0<\eta \ll 1$. In view of Lemma \ref{lem:energy} we can choose $\eta$ small enough  in such a way that 
\begin{equation}\label{val.neg}
H(a(0),b(0)) <0.
\end{equation}
By definition of $\K_\eta$ we have $a_0(0)> 0$ and without loss of generality we can also assume that $b_0(0)\geq 0$.

 \subsection{Estimates for $J$ small} 
 
 \begin{lemm} \ph \label{lemC2}
For $\eta$ small enough, as long as $J(t)\leq \eta^{5}$, there exists   $C_{1}=C_{1}(m)>0$ such that 
\begin{equation}\label{conf}
\eta/2\leq a_{0}(t)\leq C_{1},\quad |b_{0}(t)|\leq C_{1}.
\end{equation} 
\end{lemm}

\begin{proof}
Assume that $J\leq \eta^{5}$. First we prove that $a_0$ cannot vanish: Indeed if it was the case, by~\eqref{HamiJ} we would have $H\geq 0$ which is in contradiction with \eqref{val.neg}. Next we prove that there exists~$C>0$ so that 
\begin{equation}\label{claim}
\frac1{2(p+1)}\int_{M}\big(a_{0}+U\big)^{2p+2}\text{d}x\geq  \frac 1{4(p+1)} \,a^{2p+2}_{0}-CJ^{p+1}.
\end{equation}
Write the binomial expansion
\begin{equation}\label{binom}
\int_{M}\big(a_{0}+U\big)^{2p+2}\text{d}x=a^{2p+2}_{0}+\sum^{2p+2}_{j=1}\binom {2p+2} {j}a^{2p+2-j}_{0}\int_{M}U^{j}\text{d}x,
\end{equation}
Apply  Lemma \ref{lemsobo} and use  the Young inequality
\begin{equation*} 
c_{1}c_{2}=(\eps c_{1})(\eps^{-1}c_{2})\leq \frac{\eps^{q}}{q}c_{1}^{q}+\frac{1}{r\eps^{r}}c_{2}^{r},\quad c_{1},c_{2}\geq 0,\quad \eps>0, \quad \frac1{q}+\frac1{r}=1,
\end{equation*} 
 to each term of the sum with $\eps>0$ small enough, and get
\begin{equation*}
\sum^{2p+2}_{j=1}\binom {2p+2} {j}a^{2p+2-j}_{0}\Big|\int_{M}U^{j}\text{d}x\Big| \leq C\sum^{2p+2}_{j=1}a^{2p+2-j}_{0}J^{j/2}\leq  \frac12 a^{2p+2}_{0}+CJ^{p+1},
\end{equation*}
which together with \eqref{binom} yields \eqref{claim}. Now, for $\eta>0$ small enough, $CJ^{p+1}\leq J$ and by \eqref{HamiJ} and~\eqref{claim} we obtain
\begin{equation}\label{borne}
 \frac12(b^{2}_{0}-m^{2}a^{2}_{0})+\frac 1 {4(p+1)}\,a^{2p+2}_{0}\leq H=-\frac12m^{2}\eta^{2}+\mathcal{O}(\eta^{3})\leq -\frac14m^{2}\eta^{2},
\end{equation}
from which we deduce  that   $|b_{0}|\leq C_{1}$ and $0<a_{0}\leq C_{1}$ for some constant $C_1$ depending on $m$.  Furthermore,  the  bound \eqref{borne} implies $-m^{2}a^{2}_{0}/2\leq H \leq -\frac14m^{2}\eta^{2}$ which completes the proof.
\end{proof}

In the following we use the additional notations $I$ and $\hat r$:
\begin{equation*} 
I:=\frac12\sum_{k=1}^{+\infty}\Big[\big(\lambda_{k}^{2}-m^{2}+(2p+1)a^{2p}_{0}\big)a^{2}_{k}+b^{2}_{k}\Big],
\end{equation*}
\begin{equation*} 
\hat r(t)=\frac{1}{2(p+1)}\int_{M}\big((a_0+U)^{2p+2}-a_0^{2p+2}-(2p+1)a_0^{2p}U^2\big)\text{d}x.
\end{equation*}
So that $H$ reads
$$H=\frac12 b_0^2+\frac12 f(a_0)+I+\hat r.$$
Observe that as long as $J\leq C\eta^{5}$, since $a_{0}$ is bounded (see \eqref{conf}), there exists $K_{0}>0$ so that for all $(a_{k},b_{k})_{k\geq 1}\in h^{1} \times \ell^2 $, 
\begin{equation}\label{equiv}
J\leq I\leq K_{0}J.
\end{equation}

\begin{lemm}\ph \label{Lem:I}
As long as $J(t)\leq \eta^{5}$,  we have the a priori bound
\begin{equation}\label{Ipoint}
|\dot{I}|\leq Ca^{2p-1}_{0}|\dot{a}_{0}|I+CI^{3/2}.
\end{equation}
\end{lemm}

\begin{proof} We recall the definition \eqref{def.pk} of the $q_k$, and that then the equations on $(a_k,b_k)$ for $k\geq 1$ read
\begin{equation*}
\dot{a}_{k} =b_{k},\quad \dot{b}_{k} =-\big(\lambda_{k}^{2}-m^{2}+(2p+1)a^{2p}_{0}\big)a_{k}-q_{k}.
\end{equation*}
With this system, we compute  
\begin{equation}\label{b1}
|\dot{I}|=\big|p(2p+1)a^{2p-1}_{0}\dot{a}_{0}\sum_{k=1}^{+\infty}a^{2}_{k}-\sum_{k=1}^{+\infty}b_{k}q_{k}\big|\leq p(2p+1)a^{2p-1}_{0}|\dot{a}_{0}|\sum_{k=1}^{+\infty}a^{2}_{k}+\|b\|_{\ell^{2}(\N^{*})}\|q\|_{\ell^{2}(\N^{*})}.
\end{equation}
We have $ \|b\|_{\ell^{2}(\N^{*})} \leq CI^{1/2}$. Then by Parseval, \eqref{equiv} and the fact that $|a_{0}|\leq C$
\begin{eqnarray}\label{b2}
\|q\|_{\ell^{2}(\N^{*})}&=&\big\|(a_{0}+U)^{2p+1}-a^{2p+1}_{0}-(2p+1)a^{2p}_{0}U\big\|_{L^{2}(M)}\nonumber\\
&\leq &C \|U^{2}\|_{L^{2}(M)}+C \|U^{2p+1}\|_{L^{2}(M)}\nonumber\\
&=&C \|U\|^{2}_{L^{4}(M)}+C \|U\|^{2p+1}_{L^{2(2p+1)}(M)}\nonumber\\
&\leq &CI+CI^{p+\frac12}\leq CI,
\end{eqnarray} 
where in the last line we used Lemma \ref{lemsobo}. Then \eqref{b2} together with \eqref{b1} gives the result.
\end{proof}

\begin{lemm}\label{recip}
If $(a,b)\in\X$ such that $H(a,b)=-\frac 1 2 m^2 \eta^2 +O(\eta^3)$,  $\eta\leq a_0$ and $J\leq \eta^3$, then
$$\dist_\X\Big( (a_0,b_0,0,\cdots),\K_\eta\Big)\leq C\eta^2.$$
More precisely, $(a_0,b_0,0,\cdots)$ is $\eta^2$-close to $(\tilde a_0,\tilde b_0,0,\cdots)\in \K_\eta$ defined by 
\begin{equation*}
(\tilde a_0,\tilde b_0)= \left\{\begin{array}{ll} 
(a_0,sgn(b_0)\sqrt{f(\eta)-f(a_0)})\quad &\text{if} \quad  a_0\in  [\delta, \delta'] , \\[6pt]  
(f^{-1}\big(f(\eta)-b_0^2\big), b_0)   &\text{if} \quad a_0\notin [\delta,\delta'].
\end{array} \right.
\end{equation*}
\end{lemm}

\begin{proof}
From the expression \eqref{HamiJ} of the energy, the assumption reads
$$b_0^2+f(a_0)+J+r=-m^2\eta^2+O(\eta^3).$$
On the one hand, observe that if $J\leq \eta^3$, then $r\leq c\eta^3$ : the proof is similar to the one of \eqref{b2}. On the other hand, $f(\eta)=-m^2\eta^2+O(\eta^3)$. Thus we have
\begin{equation}\label{eq:energy}
b_0^2+f(a_0)=f(\eta)+O(\eta^3). 
\end{equation}
Recall that $(\tilde a_0,\tilde b_0,0,\cdots)$ belongs to $\K_\eta$ given that
\begin{equation}\label{keta}
\tilde b_0^2+f(\tilde a_0)=f(\eta).
\end{equation}
$\bullet$ Assume that $a_{0}\notin [\delta,\delta']$.  From \eqref{eq:energy} we deduce
 \begin{equation*}
 a_{0}=f^{-1}\Big(f({\eta})-b^{2}_{0}+O(\eta^3)\Big).
 \end{equation*}
 Since $a_0\geq\eta$,  thus $|f'(a_{0})|\geq c_{0}\eta$ and we can apply the mean value theorem to get  
\begin{equation*}
|\tilde a_0-a_{0}|\leq C\eta^3\sup_{}|(f^{-1})'|\leq C\eta^{2}.
\end{equation*}
$\bullet$  In the region $a_{0}\in  [\delta,\delta']$, we define $g(b_{0})=b^{2}_{0}$. Then  we have $|g'(b_{0})|\geq c$ and we can perform a similar argument to get $|\tilde b_{0}-b_{0}|\leq C\eta^{3}$.
\end{proof}

\subsection{First loop : proof of Theorem \ref{thm4} $(i)$}\label{sect24}
This part is devoted to the proof of $(i)$ of Theorem~\ref{thm4}. More precisely, we prove a stronger version of it in Lemma \ref{LemT}, which will then be useful to prove $(ii)$ of the theorem (in the next part). 

Recall that in all this section we have  fixed an  initial condition such that   
\begin{equation}\label{dista}
\dist_\X\Big( (a(0),b(0)),\K_\eta\Big)\leq \eta^3
\end{equation}
  and $0<\eta \ll 1$ such that \eqref{val.neg} is satisfied.

\begin{lemm}\label{lem:bootstrap}
There exists $K>1$ independent of $\eta>0$ such that: as long as $b_0(t)\geq 0$   and $|t|\leq \eta^{-2}$, $J$ satisfies 
$$J(t)\leq KJ(0)\leq \eta^{5}.$$ 
\end{lemm}

\begin{proof} 
We prove it by a bootstrap argument. As long as $J(t)\leq \eta^{5}$ and $b_0(t)\geq 0$,   the estimate~\eqref{Ipoint} holds true and by Gronwall we get
\begin{eqnarray*}
J(t)\leq I(t)&\leq& I(0)\e^{\int_{0}^{|t|}(\dot{a}_{0}(s)a^{2p-1}_{0}(s)+\eta^{5/2})\text{d}s}\\
 &\leq& I(0)\e^{(a_0(t)^{2p}-a_0(0)^{2p}+|t|\eta^{5/2})}\\
 &\leq &I(0)\exp{\big(2C^{2p}_{1}(m)+1\big)},
\end{eqnarray*}
where $C_{1}(m)>0$ is defined in \eqref{lemC2}.

Recall that by \eqref{conf}, $a_0$ is bounded as long as $J(t)\leq \eta^{5}$, and that $I(0)\leq K_{0}\eta^6$ by \eqref{dista} (recall the definition of $K_{0}$ in \eqref{equiv}). Therefore, under the assumptions  that $b_0(t)\geq 0$   and $|t|\leq \eta^{-2}$ we get  that as long as $J(t)\leq \eta^{5}$, we have $J(t)\leq K_{1}I(0)$, which gives the result for $\eta$ sufficiently small, since $I(0)\leq K_{0}J(0)$.
\end{proof}

\begin{lemm}\label{lem:tau}
Let $\alpha_0$ be a real with $0<\alpha_{0}\leq\delta$. We suppose that the initial conditions $(a(0),b(0))$ satisfy $0<a_0(0)<\alpha_0$ and $b_0(0)\geq 0$.
Then there exist times $\tau(\alpha_0)$ and $\tau_\eta$ such that
$$a_0(\tau(\alpha_0))=\alpha_0, \;a_0(\tau_\eta)=\delta \quad \text{and} \quad b_0(t)>0 \;\text{ for }\;t\in]0,\tau_\eta].$$
Moreover, we have the estimates
$$\tau(\alpha_0)\leq\tau_{\eta} \leq C\ln \frac1{\eta,}$$
and if $b_0(0)=0$,
$$c\,\ln \frac1{\eta}\leq \tau_\eta.$$
\end{lemm}

\begin{proof}

$\bullet$ Firstly, we can go  back to the particular case $b_0(0)=0$ : for that, we prove that there exists a time $t\in]0,\eta^{-2}[$ such that $b_0(-t)=0$. We proceed by contradiction : suppose that $b_0(-t)>0$ for all $t\in[0,\frac{4b_0(0)}{m^2\eta}]$. We know from Lemma \ref{lem:bootstrap}, that for all those $t$, $J(-t)\leq \eta^5$. Thus $q_0(-t)\leq C\eta^5$, and by~\eqref{conf}, we have that  $a_0(-t)>\eta/2$. Hence,
$$\dot b_0(-t)=-f'(a_0(-t))-q_0(t)> -f'(\frac{\eta}{2})+O(\eta^5)=m^2\frac\eta2+O(\eta^3)>m^2 \frac\eta3.$$
 Then for all $t$ in $[0,\frac{4b_0(0)}{m^2\eta}]$,
 $$b_0(-t)<b_0(0)-t m^2\frac{\eta}{3}$$
holds. In particular, we obtain $b_{0}(-\frac{4b_0(0)}{m^2\eta})<0$ : the contradiction.

$\bullet$ So for all  initial conditions satisfying the assumption of the lemma, there exists a negative time\,$-t$ such that $b_0(-t)=0$ and $J(-t)\leq \eta^5$.

 In the following we only consider the case $$a_{0}(0)=\eta \quad \text{and} \quad b_0(0)=0.$$
 Observe that   the upper bound we will get for $\tau_\eta$ in this case will be an upper bound in the general case. 

\vspace{1ex}


$\bullet$ 
Let $\phi$ be the function defined by
\begin{equation*}
\eta^{5}\phi(t)=-2(J(t)-J(0))-2(r(t)-r(0)).
\end{equation*}
From the preservation of energy, see \eqref{HamiJ} and \eqref{heta}, we get that $\phi$ satisfies
\begin{equation}\label{phi}
b^{2}_{0}+f(a_{0})=f(\eta)+\eta^{5}\phi(t).
\end{equation}
Given that $b_0(0)=0$ and $\dot{b}_0(0)>0$ for $\eta$ sufficiently small,   $b_0(t)>0$ for some times, and  as long as $b_{0}(t)\geq 0$, $b_0(t)$ reads
\begin{equation}\label{eq:b0}
\dot{a}_0(t)=b_{0}(t)=\sqrt{f(\eta)-f(a_{0}(t))+\eta^{5}\phi(t)}.
\end{equation}
In particular, as long as $b_0(t)>0$, by integration and the change of variables $\a=a_{0}(t)$ we obtain
\begin{equation}\label{eq:t}
t=\int_{0}^{t}\frac{\dot a_{0}(s)\text{d}s}{\sqrt{f(\eta)-f(a_{0}(s))+\eta^{5}\phi(s)}}=\int_{\eta}^{a_0(t)}\frac{\text{d}\a}{\sqrt{f(\eta)-f(\a)+\eta^{5}\phi(a^{-1}_{0}(\a))}}.
\end{equation}

$\bullet$ Let us prove that, as long as $J(t)\leq \eta^{5}$
\begin{equation}\label{quad}
|\phi(t)|\leq C\min{(1,t^{2})}.
\end{equation}
Since $J\leq \eta^{5}$ and by \eqref{equiv} together with \eqref{b2}, we have $|\phi(t)|\leq C$. Then, since $b_{0}(0)=0$ and $\dot{a}_{0}(0)=b_{0}(0)=0$, from \eqref{phi} we infer that $\dot{\phi}(0)=0$. To obtain \eqref{quad}, it is therefore enough to prove that $|\ddot{\phi}|\leq C$. This is done by computing $\ddot{J}$ and $\ddot r$ with the relations \eqref{systn}, and with Cauchy-Schwarz  we can check that   $|\ddot J |  \leq C\eta^{5}$ and  $|\ddot r|\leq C\eta^{5}$.\\

$\bullet$ We now show that there exist $c,C>0$ so that as long $J(t)\leq \eta^{5}$, $a_0(t)\leq \delta$ and $b_0(t)\geq 0$,
\begin{equation}\label{ineq1}
c(a_0^{2}-\eta^{2})\leq f(\eta)-f(a_0)+\eta^{5}\phi(t)\leq C(a_0^{2}-\eta^{2}).
\end{equation}
Observe that for all $\eta\leq \alpha\leq\delta$, 
 $$(m^2-\delta^{2p})(\alpha^{2}-\eta^{2})\leq f(\eta)-f(\a)\leq m^2(\alpha^{2}-\eta^{2})$$
holds, and recall that $\delta< m^{\frac{1}{p}}$. It is then enough to show that $\eta^{5}\phi(t)$ is a negligible amount when $\eta$ small. Recall the form \eqref{syst0} of the system :
\begin{equation*}
\ddot{a}_{0} =\dot{b_0}=m^{2}a_0-a_0^{2p+1}-q_{0}=-f'(a_0)-q_0(t).
\end{equation*}
On one hand, with the same calculations as in the proof of Lemma \ref{Lem:I}, we obtain that as long as $J(t)\leq \eta^{5}$, we have $|q_0(t)|\leq \eta^{5}$. On the other hand, by a study of the variations of $f'$, we get that for $\eta\leq\alpha\leq\delta$, if $\eta$ small enough then $f'(\alpha)\leq -c\eta$. We deduce that as long as $J(t)\leq \eta^{5}$ and $a_0\leq\delta$, we have $\ddot a_{0}\geq c\eta$ and by Taylor, $\dis |a_{0}(t)-a_{0}(0)|\geq c\eta t^{2}$. This in turn implies that when moreover $a_0^{-1}$ is well-defined (for $b_{0}(t)>0$)
\begin{equation*}
|a^{-1}_{0}(\a)|=|a^{-1}_{0}(\a)-a^{-1}_{0}(\eta)|\leq \frac{C}{\eta^{1/2}}|\a-\eta|^{1/2} 
\end{equation*}
holds. Then by \eqref{quad}, we get that
\begin{equation*}
\eta^{5}\big|\phi(t)\big|\leq C\eta^{5}\eta^{-1}|a_0-\eta|\leq C \eta^{3}(a_0^{2}-\eta^{2}),
\end{equation*}
which proves \eqref{ineq1}.

$\bullet$ Firstly, from \eqref{ineq1} with \eqref{eq:b0}, we get that as long as $J(t)\leq \eta^{5}$ and $a_0(t)\leq\delta^2$, we have $b_0(t)>0$. 
Secondly, with \eqref{ineq1} together with \eqref{eq:t}, we can compute an estimate of $t$ in term of $a_0(t)$. We will see then that $a_0(t:=\tau_\eta)=\delta$ holds for a time $\tau_\eta$ smaller than $\eta^{-2}$, so that by Lemma \ref{lem:bootstrap} we are still in the regime $J(t)\leq \eta^{5}$. Precisely, observe that
$$\dis   \int_{\eta}^{\alpha_0}\frac{\text{d}\a}{\sqrt{\a^{2}-\eta^{2}}} =\ln\left(\frac{\alpha_0}{\eta}+\sqrt{\frac{\alpha_0^2}{\eta^2}-1}\right).$$
Then \eqref{ineq1} together with \eqref{eq:t} gives the estimate claimed in the lemma, and in particular $\tau_\eta\leq \eta^{-2}$ for $\eta$ sufficiently small. 
\end{proof}

\begin{lemm}\label{lem:T'}
For any initial condition $(a(0),b(0))$ such that $b_0(0)\geq 0$, there exists a time $T^1_\eta$ such that
$$b_0(T^1_\eta)=0, \quad \text{and}\quad b_0(t)>0 \text{ for all }t\in]0,T^1_\eta[.$$
Moreover, if $a_0(0)<\delta$, then there exist $C$ such that
$$0<T^1_\eta-\tau_\eta \leq C,$$
where $\tau_{\eta}$ is given by Lemma \ref{lem:tau} and if $a_0(0)\geq \delta$, then
$$T^1_\eta\leq C.$$
\end{lemm}
\begin{proof}
Recall that $\delta\in]0,m^{\frac1p}[$, and $\delta'$ is the unique real of $]m^{\frac1p},(p+1)^{\frac1{2p}}m^{\frac1p}[$ such that $f(\delta')=f(\delta)$. In the proof we consider the case $a_0(0)<\delta$. To prove the lemma, we proceed in two steps, here is a summary of them :
\begin{enumerate}[(i)]
\item We show that there exists a time $\tau^1_\eta>\tau_\eta$ such that $b_0(\tau^1_\eta)=b_0(\tau_\eta)$, $a_0(\tau^1_\eta)-\delta'\leq C\eta^{5}$ and $b_0(t)>b_0(\tau_\eta)$ for $t\in]\tau_\eta,\tau^1_\eta[$. Moreover, $\tau^1_\eta-\tau_\eta\leq C$ holds.
\item Then we show the existence of $T^1_\eta>\tau^1_\eta$ such that $b_0(T^1_\eta)=0$, $|a_0(T^1_\eta)-h(0)|\leq C\eta^2$, and $b_0(t)>0$ for all $t\in]\tau^1_\eta,T^1_\eta[$. Moreover, $T^1_\eta-\tau^1_\eta\leq C$ holds. 
 \end{enumerate}
 If $a_0(0)\in]\delta,\delta'[$, then in step (i) the proof is the same by considering $a_0(0)$ instead of $a_0(\tau_\eta)$. If $a_0(0)>\delta'$, then step (ii) is sufficient, considering here $a_0(0)$ instead of $a_0(\tau^1_\eta)$.

\vspace{1ex}

 {\bf Step (i). } From Lemma \ref{lem:tau} we know that $a_0(\tau_\eta)=\delta$. We then obtain that $b'_0(\tau_\eta), b_0(\tau_\eta)\geq c_0>0$ ($c_0$ depends only on $\delta$ which is fixed). In particular, $b_0(t)\geq b_0(\tau_\eta)$ for small times  $t\geq \tau_\eta$. Let us prove this step, by contradiction : suppose that for all $t $ in $]\tau_\eta, \tau_\eta+(h(0)-\delta)/c_0]$, we have $b_0(t)>b_0(\tau_\eta)$.  Then for all those $t$, 
$$a_0(t)>\delta+c_0\cdot(t-\tau_\eta).$$
Then we have $a_0(\tau_\eta+(h(0)-\delta)/c_0)>h(0)$, and thus there exists $t\in ]\tau_\eta, \tau_\eta+(h(0)-\delta)/c_0]$ such that $f(a_0(t))=f(h(0))=0$. Hence, for this $t$, the assumption $b_0(t)>b_0(\tau_\eta)$ is in contradiction with the energy preservation \eqref{phi} for $\eta$ sufficiently small, given that $t$ is in the regime $|\phi(t)|\leq C$. Finally, there exists $\tau^1_\eta$ such that $0<\tau^1_\eta-\tau_\eta\leq C$ and $b_0(\tau^1_\eta)=b_0(\tau_\eta)$.   

To get the approximation of $a_0(\tau^1_\eta)$, on the one hand observe that necessarily there exists a time $t\in ]\tau_\eta,\tau^1_\eta[$ such that $b'_0(t)=0$. From \eqref{syst0}  we thus have 
$|f'(a_0(t))|\leq C\eta^{5}$, which means that $|a_0(t)-m^{\frac1p}|\leq C'\eta^{5}$. On the other hand, from the preservation of the energy \eqref{phi}, given that we are still in the regime $|\phi(t)|\leq C$, we obtain that $|f(a_0(\tau^1_\eta))-f(\delta)|\leq C\eta^{5}$. Then, given that $a_0(\tau^1_\eta)>a_0(t)$, we get that for $\eta$ sufficiently small, $|a_0(\tau^1_\eta)-\delta'|\leq C\eta^{5}$.
\vspace{1ex}

{\bf Step (ii). } Let us prove this step by contradiction : suppose that for all $t $ in $]\tau^1_\eta, \tau^1_\eta+2 b_0(\tau_\eta)/f'(\delta)]$, we have $b_0(t)>0$.  Then for all those $t$, $a_0(t)>a_0(\tau^1_\eta)$ and thus
$$b'_0(t)<-f'(\delta')+O(\eta^{5})<\frac12 f'(\delta) \quad \text{for }\eta\text{ small}.$$
Then we have $b_0(\tau^1_\eta+2 b_0(\tau_\eta)/f'(\delta))<0$, and get the contradiction. Thus there exists $T^1_\eta$ such that $b_0(T^1_\eta)=0$ and $0<T^1_\eta-\tau^1_\eta\leq C$. From the energy preservation \eqref{phi} and given that $a_0(T^1_\eta)>a_0(\tau^1_\eta)>m^{\frac1p}$ and $f'(h(0))>0$, we get the estimate on $a_0(T^1_\eta)$.
\end{proof}

\begin{lemm}\label{LemT}
There exist some constants $c_0,C_0,C>0$ such that for any initial condition $(a(0),b(0))$ satisfying 
$$dist\big((a_0(0),b_0(0)),\K_\eta)\big)\leq \eta^3, \quad J(0)\leq \eta^5,$$
there exists a time $T_\eta$ such that
$$|(a_0(T_\eta),b_0(T_\eta))-(a_0(0),b_0(0))|\leq C\eta^2, \quad J(T_\eta)\leq 2KJ(0)\leq \eta^5$$
and 
$$c_0\ln \frac1{\eta}\leq T_{\eta} \leq C_0\ln \frac1{\eta}.$$
Precisely, we can define uniquely $T_\eta$ as the time of first return to the set 
$$\big\{(a,b) \;| \; a_0=a_0(0), \;sgn(b_0)=sgn(b_0(0))\big\}, \quad \text{if }\; a_0(0)\in[\delta,\delta']$$ 
or
$$\big\{(a,b)\; |\; b_0=b_0(0),\; sgn (a_0-m^{\frac1p})=sgn(a_0(0)-m^{\frac1p})\big\}, \quad \text{if } \;a_0(0)\notin[\delta,\delta'].$$
\end{lemm}

\begin{proof}
Let us explain the strategy of the proof. We follow the periodic orbit with initial $(a(0),b(0))$. When we are $\delta-$close to 0 we use the arguments of   Lemma  \ref{lem:tau}, and when $a_{0}\geq \delta$, we use the arguments of Lemma \ref{lem:T'}.




Notice that by Lemma \ref{recip} and the precise definition of $T_{\eta}$  we know that $(a_{0}(0),b_{0}(0))$ and  $(a_{0}(T_{\eta}),b_{0}(T_{\eta}))$ are $\eta^{2}$ close to each other. 
 \end{proof}

\subsection{Many loops : proof of Theorem \ref{thm4} $(ii)$}

We would like to take $(a(T_\eta),b(T_\eta))$ as an initial condition in Lemma \ref{LemT} to iterate the process. But by Lemmas \ref{LemT} and \ref{recip}, we only know that  
$$dist\big((a_0(T_\eta),b_0(T_\eta)),\K_\eta\big)\leq\eta^2, \quad J(T_\eta)\leq K\eta^6.$$ 
This is a priori too weak to apply Lemma \ref{LemT}. But we can choose to apply the Lemma \ref{LemT} with $\eta'$ such that $(a_0(T_\eta),b_0(T_\eta))\in\K_{\eta'}$, which gives the same estimates.

We see then that we can iterate the argument as long as  $J\leq   \eta^{5}$, namely $N$ times (each time corresponds to a loop) with 
$$(2 K)^{N}\eta^6= \eta^{5}\quad \text{i.e. } N=c\ln \frac1{\eta}.$$ This gives that  for $|t|\leq c\big(\ln\frac1{\eta}\big)^{2}$, $J\leq \eta^{5}$ and we can iteratively apply Lemma \ref{LemT}.


\begin{thebibliography}{99}

\bibitem{BamGre}
D.~Bambusi and B.~Gr{\'e}bert.
\newblock  Birkhoff normal form for {PDE}s with tame modulus.
\newblock{\em Duke Math. J.} 135 (2006), 507--567.

\bibitem{EK2}
L.H. Eliasson and  S.B. Kuksin.
\newblock On reducibility of Schr\"odinger equations with quasiperiodic in time potentials.  
\newblock{\em Comm. Math. Phys.}  286  (2009),  no. 1, 125--135.

\bibitem{GJT}
B.~Gr{\'e}bert, T. J\'ez\'equel and L. Thomann.
\newblock  Dynamics of Klein-Gordon on a compact surface near a homoclinic orbit.
\newblock{\em To appear in DCDS}.

\bibitem{GT}
B.~Gr{\'e}bert and L. Thomann.
\newblock  KAM for the Quantum Harmonic Oscillator.
\newblock{\em Comm. Math. Phys.} 307 (2011), 383--427.

\bibitem{kuk}
S. B. Kuksin.
 \newblock Nearly integrable infinite-dimensional Hamiltonian systems. 
 \newblock{\em  Lecture Notes in Mathematics, 1556.} Springer-Verlag, Berlin, 1993.
 
 \bibitem{Posch}
J. P\"oschel.
\newblock Quasi-periodic solutions for a nonlinear wave equation. 
\newblock{\em Comment. Math. Helv.}  71 (1996), no.~2, 269--296.

      \bibitem{Wayne}
 C. E. Wayne.
 \newblock Periodic and quasi-periodic solutions of nonlinear wave equations via KAM theory.
 \newblock{\em   Comm. Math. Phys. } 127 (1990), no. 3, 479--528.
 \end{thebibliography}
\end{document}